\pdfoutput=1
\documentclass[reqno,12pt]{amsart}

\usepackage{enumerate}
\usepackage[margin=1in]{geometry}
\usepackage{ifpdf}
\usepackage{amsmath}
\usepackage{amsfonts}
\usepackage{amssymb}
\usepackage{amsthm}
\usepackage[hyperfootnotes=false]{hyperref}
\usepackage{setspace}
\usepackage{amsrefs}
\usepackage{graphicx}
\usepackage{nicefrac}

\newcommand{\ignore}[1]{}

\newcommand{\abs}[1]{\left\lvert {#1} \right\rvert}

\newcommand{\sabs}[1]{\lvert {#1} \rvert}
\newcommand{\snorm}[1]{\lVert {#1} \rVert}

\newcommand{\Z}{{\mathbb{Z}}}
\newcommand{\N}{{\mathbb{N}}}

\newcommand{\sF}{{\mathcal{F}}}

\newcommand{\sH}{{\mathcal{H}}}

\newtheorem{thm}{Theorem}[section]

\newtheorem{prop}[thm]{Proposition}

\newtheorem{cor}[thm]{Corollary}

\newtheorem{lemma}[thm]{Lemma}

\theoremstyle{definition}
\newtheorem{defn}[thm]{Definition}

\theoremstyle{remark}
\newtheorem{remark}[thm]{Remark}

\author{Ji\v{r}\'i Lebl}
\thanks{The first author was in part supported by NSF grant DMS 0900885.}
\address{Department of Mathematics, University of California
at San Diego, La Jolla, CA 92093-0112, USA}
\email{jlebl@math.ucsd.edu}

\author{Han Peters}
\thanks{The second author was supported by a SP3-People Marie Curie Actionsgrant in the project Complex Dynamics (FP7-PEOPLE-2009-RG, 248443)}
\address{Korteweg De Vries Institute for Mathematics, University of Amsterdam, Science Park 904,
1098 XH Amsterdam,
The Netherlands,}
\email{h.peters@uva.nl}

\dedicatory{Dedicated to Professor John P.\ D'Angelo on the occasion of his
sixtieth birthday.}

\date{September 26, 2011}

\ifpdf
\hypersetup{
  pdftitle={Polynomials constant on a hyperplane and CR maps of spheres},
  pdfauthor={Jiri Lebl and Han Peters}
}
\fi

\title[Polynomials constant on a hyperplane and maps of spheres]%
{Polynomials constant on a hyperplane and CR maps of spheres}

\begin{document}

%\doublespace

\begin{abstract}
We prove a sharp degree bound for polynomials constant on a hyperplane with a
fixed number of nonnegative distinct monomials. This bound was conjectured
by John P.\ D'Angelo, proved in two dimensions by D'Angelo, Kos and Riehl and
in three dimensions by the authors. The current work builds upon these
results to settle the conjecture in all dimensions. We also give a complete
description of all polynomials in dimensions $4$ and higher for which the
sharp bound is obtained.
The results prove the sharp degree bounds for monomial CR mappings of
spheres in all dimensions.
\end{abstract}

\keywords{Polynomials constant on a hyperplane, CR mappings of spheres,
monomial mappings, degree estimates, Newton diagram}
\subjclass[2000]{14P99, 05A20, 32H35, 11C08}

\maketitle

%\enlargethispage{\baselineskip}

%%%%%%%%%%%%%%%%%%%%%%%%%%%%%%%%%%%%%%%%%%%%%%%%%%%%%%%%%%%%%%%%%%%%%%%%%%%%%

\section{Introduction} \label{section:intro}

We are interested in the degree and the number of monomials for polynomials
with nonnegative coefficients that are constant on the hypersurface
$$
x_1 + \cdots + x_n = 1.
$$
We denote by $N(p)$ the number of nonzero coefficients of a polynomial $p$.
Our main result is the following.

\begin{thm} \label{thm:monomialdegbnd}
Let $n \geq 2$ and let $p$ be a polynomial with non-negative coefficients of
degree $d$ in $n$ dimensions such that $p(x) = 1$ whenever
$x_1+\cdots + x_n = 1$. Then the following inequality holds and is
sharp
\begin{equation} \label{eqn:degbnd}
d \leq
\begin{cases}
2N(p)-3 & \text{if $n =2$} ,
\\
\frac{N(p)-1}{n-1} & \text{if $n \geq 3$} .
\end{cases}
\end{equation}
\end{thm}
For $n \ge 3$ sharp means that for every $n$ and every $d$ there exist examples for which equality holds. When $n = 2$ equality can only hold for odd degree, for even degree there exist examples so that the right hand side is exactly one larger.

Inequality \eqref{eqn:degbnd} was conjectured by D'Angelo~\cite{DKR}. The $2$-dimensional case was
proved by D'Angelo, Kos and Riehl in \cite{DKR}. The $3$-dimensional case
was proved by the authors in \cite{LP}. In this article we use ideas of
the $2$-dimensional proof for an induction argument that uses the $3$-dimensional result as a base to obtain the result in
dimensions $4$ and higher.

In dimensions $4$ and higher we also give a complete description of all
polynomials for which inequality \eqref{eqn:degbnd} is an equality. We note
that this description does not hold in dimensions $2$ and~$3$.

It is instructive to note that no degree bound can hold when $n=1$; the
polynomial $p(x) = x^d$ gives 1 whenever $x=1$, but $N(p) = 1$ while $d$ can
be arbitrarily large.

We end the introduction with a brief description of the history of the
problem. The motivation of the problem comes from the study of proper
holomorphic maps between balls in different dimensions. We denote by
$\mathbb{B}^n$ the unit ball in $n$ complex dimensions. Faran
\cite{Faran:B2B3} showed that a proper holomorphic map from $\mathbb{B}^2$
to $\mathbb{B}^3$ that is sufficiently smooth on the boundary is spherically
equivalent to a monomial map of degree at most $3$. It was shown by
Forstneri\v{c} \cite{Forstneric} that any proper holomorphic map
from $\mathbb{B}^n$ to $\mathbb{B}^N$, with $n \geq 2$, is rational if the map is
sufficiently smooth up to the boundary, and for fixed dimensions $n$ and $N$
the degree is bounded from above by a constant depending only on $n$ and $N$.

The bound on the degree in terms of $n$ and $N$ has been studied by various
authors.  As mentioned above, Forstneri\v{c}~\cite{Forstneric}
proved that a bound exists.  The best currently known bound in the
general rational case was proved by
Meylan~\cite{Meylan} in $n=2$ and extended to $n \geq 3$ by
D'Angelo-Lebl~\cite{DL:families}.
D'Angelo has conjectured that the best possible bound for $n = 2$ is
given by $d \leq 2N-3$, and for $n \geq 3$ it is given by $d \leq
\frac{N-1}{n-1}$. The reader will notice that these are exactly the same
inequalities as those in Theorem \ref{thm:monomialdegbnd}.
When $N$ is small compared to $n$, then the conjectured bound is known to hold
in the general rational case.
When $n=2$ and $N=3$ it follows from Faran's work~\cite{Faran:B2B3} that
$d \leq 3$.  When
$n \geq 3$ and $N \leq 3n-4$, then $d \leq 2$ by the work of
Huang and Ji~\cite{HJ01} and Huang, Ji, and Xu~\cite{HJX:gap}.

Suppose that the proper holomorphic map $f \colon \mathbb{B}^n \to \mathbb{B}^N$ is a monomial map, that is, every coordinate function is a monomial. Of course in this case the map automatically extends to the boundary, and the properness of the map means exactly that
\begin{equation}
\snorm{f(z_1, \ldots, z_n)}^2 = 1 \; \; \; \text{when} \; \; \; \sabs{z_1}^2
+ \cdots \sabs{z_n}^2 = 1,
\end{equation}
which we can write as
\begin{equation}
\sabs{f_1(z)}^2 + \cdots + \sabs{f_N(z)}^2 = 1 \; \; \; \text{when} \; \; \;
\sabs{z_1}^2 + \cdots \sabs{z_n}^2 = 1.
\end{equation}
Since every coordinate function $f_j$ is a monomial, we get that
$\sabs{f_1(z)}^2 + \cdots + \sabs{f_N(z)}^2$ is a real polynomial $p$ in the
variables $x_1, \ldots x_n$, where $x_j = \sabs{z_j}^2$. The polynomial $p$ has
non-negative coefficients and (at most) $N$ distinct nonzero coefficients.
Theorem \ref{thm:monomialdegbnd} therefore implies the following result.

\begin{cor}
Let $n, N \in \mathbb{N}$, with $n \geq 2$ and suppose that $f \colon
\mathbb{B}^n \to \mathbb{B}^N$ is a monomial map that is proper and of degree $d$. Then the following inequality holds and is sharp
\begin{equation}
d \leq
\begin{cases}
2N-3 & \text{if $n =2$},
\\
\frac{N-1}{n-1} & \text{if $n \geq 3$} .
\end{cases}
\end{equation}
\end{cor}

We note again that this result was proved by D'Angelo, Kos and Riehl in
\cite{DKR} for $n= 2$ and by the authors in \cite{LP} for $n=3$.

While monomial mappings may seem like a special case, we would like to
note that all known examples of proper rational maps between balls are
homotopic to a monomial map.  Furthermore, the first author proved
in~\cite{Lebl:deg2} that all degree two proper maps between balls are equivalent
to monomial maps.

{\emph{Acknowledgement: The authors would like to thank John D'Angelo for
introducing this problem to us in 2005, and for his help and guidance since
then.  We hope that this resolution of his conjecture in the monomial case
is a fitting birthday present.
Some of the work leading up to these results has happened during
workshops at MSRI and AIM, and the authors would like to thank both
institutes.  The authors would also like to thank the referee for useful
suggestions on improving the exposition.}}

%%%%%%%%%%%%%%%%%%%%%%%%%%%%%%%%%%%%%%%%%%%%%%%%%%%%%%%%%%%%%%%%%%%%%%%%%%%%%

\section{Whitney polynomials} \label{section:whitney}

Denote by $\sH$ the set of polynomials $p$ with non-negative coefficients
that satisfy $p=1$ on the hyperplane $x_1 + \cdots + x_n = 1$.
We start by giving a method of constructing examples in $\sH$.
From now on we write $s = x_1 + \cdots + x_n$.  Of course $s$ itself
lies in $\sH$.  Let $p \in \sH$ be a polynomial of degree $d$, and suppose
that $m$ is a monomial with a positive coefficient such that $p - m$ still
has only non-negative coefficients. Then we can replace $m$ by $s$ times
$m$, in other words,
\begin{equation}
p - m + s \cdot m \in \sH.
\end{equation}

Note that if $m$ has degree $d$, and the coefficient of $m$ is as large as
possible, then $p - m + sm$ has degree $d+1$, and has exactly $n-1$ nonzero
coefficients more than $p$. So starting with $s$, we can repeatedly replace
one of the terms of maximal degree with $s$ times that term, to obtain many
examples of any degree for which $N = n + (d-1)(n-1)$, or $N = d (n-1) + 1$.
For example when $n=3$ (calling the variables $x$, $y$, and $z$), we could
obtain a degree 3 polynomial with 7 monomials
by the following
procedure:
\begin{equation}
\begin{aligned}
& x + y + z (x + y + z) =
x + y + x z + y z + z^2 , \\
& x + y + x z (x + y + z) + y z + z^2
=
x + y + x^2 z  + x y z + x z^2 + y z + z^2 .
\end{aligned}
\end{equation}
We call polynomials obtained by following this procedure
\emph{sharp generalized Whitney polynomials}, and
note that
\begin{equation}\label{eqn:equality}
d = \frac{N - 1}{n - 1}.
\end{equation}

As we will see, this is the best possible bound for $n \geq 3$, and in fact,
for $n \geq 4$ these are the only polynomials for which there can be equality.

\begin{thm}\label{thm:equality}
Let $n \geq 4$ and let $p \in \sH$. Then \eqref{eqn:equality} holds if and only if $p$ is a sharp generalized Whitney polynomial.
\end{thm}

The following example shows that Theorem \ref{thm:equality} does not hold for $n = 3$.
\begin{equation}\label{ex:Faran3D}
p(x,y,z) = x^3 + 3x^2 z + 3xz^2 + z^3 + 3xy + 3yz + y^3.
\end{equation}
This polynomial can be obtained from the polynomial $F = x^3 + 3xy + y^3$ by
substituting $x+ z$ for $x$. Since $F = 1$ when $x+ y = 1$ one immediately
sees that $p = 1$ when $x + y + z = 1$. The polynomial $p$ has degree $3$,
has $7$ nonzero coefficients and $3$ variables, and $\frac{7-1}{3-1} = 3$.
It is clear that $p$ is not a sharp generalized Whitney polynomial, as the highest degree terms are not of the form $s \cdot m$.

In dimension $2$ the equality $d = N-1$ is not optimal, in fact $N$ can be
decreased by roughly a factor $2$. The map $F = x^3 + 3xy + y^3$ is an
example where $d = 2N - 3$.  It was shown by D'Angelo
\cite{D88} (see also \cite{Dbook}) that such examples occur for any odd degree.

The sharp polynomials in $n=2$ constructed by D'Angelo happen to be group
invariant, see \cite{Dbook}.  Sharp polynomials that are group invariant
also occur in $n=3$, for example \eqref{ex:Faran3D} above.
However, as a consequence of Theorem~\ref{thm:equality}
no sharp polynomials are group invariant when $n \geq 4$, as sharp
generalized Whitney polynomials are not group invariant.

%%%%%%%%%%%%%%%%%%%%%%%%%%%%%%%%%%%%%%%%%%%%%%%%%%%%%%%%%%%%%%%%%%%%%%%%%%%%%

\section{Newton diagrams}

As before let $s = x_1 + \cdots + x_n$ and let $p \in \sH$.  As $p-1$ is divisible by $s - 1$ we
can define the polynomial
\begin{equation}
q = \frac{p-1}{s-1}.
\end{equation}
In general we work with $q$ rather than with $p$. While the polynomial $q$ can have negative coefficients, the fact that the coefficients of $p$ are all non-negative puts serious restrictions on $q$, which we discuss in this section.

\begin{defn}
For $\alpha \in \Z^n$, we write
$\abs{\alpha} = \alpha_1 + \alpha_2 +
\cdots + \alpha_n$. Let $p-1 = (s-1)q$ as above, where $p$ is of degree $d$.
We define the function $D \colon \Z^n \to \{ 0, P, N \}$ as follows.
Write $q$ in multi-index notation
\begin{equation}
q(x) = \sum_{\alpha} c_\alpha x^\alpha .
\end{equation}
If $c_\alpha$ does not appear in the expansion we let $c_\alpha = 0$.
\begin{equation}
D(\alpha) :=
\begin{cases}
P & \text{ if $c_{\alpha} > 0$, } \\
0 & \text{ if $c_{\alpha} = 0$, } \\
N & \text{ if $c_{\alpha} < 0$. }
\end{cases}
\end{equation}
We call $D$ the \emph{Newton diagram} of $q$, and we say that
$D$ is the Newton diagram corresponding to $p$.

We call the $\alpha \in \Z^n$ \emph{points} of $D$, and we call
$\alpha$ a $0$-point if $D(\alpha)=0$, a $P$-point if $D(\alpha) = P$ and
an $N$-point if $D(\alpha) = N$.  We say that the monomial $x^{\alpha}$
is the monomial associated to $\alpha \in \Z^n$ and vice-versa.
We often identify points of $\Z^n$ with the associated monomials.

Let $\widehat{K} \subset \Z^n$ is the smallest
set such that $D^{-1}(\{P,N\}) \subset \widehat{K}$, that is $\widehat{K}$
contains all the points where $D$ is nonzero, and such that for some
$a = (a_1,\ldots,a_n) \in \Z^n$
and some $k \in \N$ we have
\begin{equation}
\widehat{K} = \{x \in \Z^n \mid x_j \geq a_j \text{ for all $j$ and } \abs{x} \leq k\}.
\end{equation}
We define the \emph{size} of $D$ as $k-\abs{a}+1$.
\end{defn}

Therefore, the Newton diagram of $q$ is an $n$-dimensional
array of $P$s, $N$s, and $0$s,
one for each coefficient of $q$.
For convenience we include negative powers in the Newton diagram, even though $D(\alpha) = 0$ any time $\alpha$
is not in the positive quadrant.  We note that if $p-1$ has a non-zero constant term then the size of $D$ is
equal to the degree of $p$.

We generally ignore points $\alpha \in \Z^n$ where $D(\alpha) = 0$.
We give a graphical representation of $D$ by drawing a lattice, and then
drawing the values of $D$ in the lattice as circles or spheres.  For convenience,
when drawing the $n=2$
lattice, we put $(0,0)$ at the origin, and then let $(0,1)$ be directed
at angle $\frac{\pi}{3}$ and $(1,0)$ at angle $\frac{2\pi}{3}$.  Similarly we depict
the diagram for $n=3$.

In the figures, we do not draw the circles and spheres corresponding to the zero
coefficients.
See Figure~\ref{fig:newton} for sample diagrams.

\begin{figure}[h!t]
\begin{center}
\begin{minipage}[b]{1.6in}%{0.39\linewidth}
\centering
\includegraphics{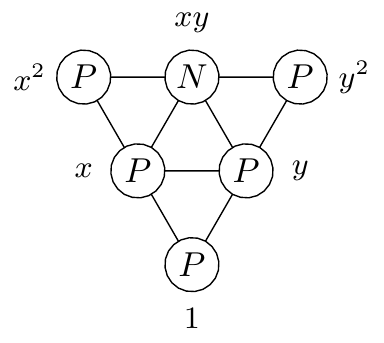}
\end{minipage}
\begin{minipage}[b]{1.6in}%{0.59\linewidth}
\centering
\includegraphics{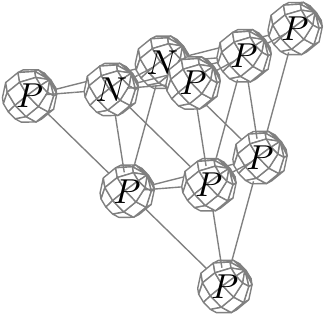}
\end{minipage}
\caption{Diagrams corresponding to the polynomial $p(x,y) = x^3 + 3xy + y^3$ on the left, and the polynomial from \eqref{ex:Faran3D} on the right.}\label{fig:newton}
\end{center}
\end{figure}

We adopt the following terminology from \cite{DKR} and \cite{LP}.

\begin{defn}
For a Newton diagram $D$ and $\alpha \in \Z^n$ we define
\begin{equation}
\operatorname{down}_k(\alpha) =
\operatorname{down}_k(\alpha_1,\ldots,\alpha_n) :=
(\alpha_1,\ldots,\alpha_{k-1},\alpha_k-1,\alpha_{k+1},\ldots,\alpha_n)
\end{equation}
to be the function that subtracts 1 from the $k$th element of $\alpha$.

We call $\alpha \in \Z^n$ a \emph{sink}, if
\begin{equation}
\begin{aligned}
& D(\alpha) = N \text{ or } 0 \\
& \quad \text{and} \\
& D\bigl(\operatorname{down}_k(\alpha)\bigr) =
P \text{ or } 0 \qquad \text{for all $k = 1,\ldots,n$} ,
\end{aligned}
\end{equation}
and at least one of $D(\alpha)$ or $D\bigl(\operatorname{down}_k(\alpha)\bigr)$ is
nonzero.

Similarly we call $\alpha \in \Z^n$ a \emph{source}, if
\begin{equation}
\begin{aligned}
& D(\alpha) = P \text{ or } 0 \\
& \quad \text{and} \\
& D\bigl(\operatorname{down}_k(\alpha)\bigr) =
N \text{ or } 0 \qquad \text{for all $k = 1,\ldots,n$} ,
\end{aligned}
\end{equation}
and at least one of $D(\alpha)$ or $D\bigl(\operatorname{down}_k(\alpha)\bigr)$ is
nonzero.

When we want to say that $\alpha$ is a sink or a source, but we do
not necessarily need to differentiate, we say $\alpha$ is a \emph{node}.
Define
\begin{equation}
\#(D) :=
\text{\# of nodes in $D$}.
\end{equation}
\end{defn}

\begin{figure}[h!t]
\begin{center}
\includegraphics{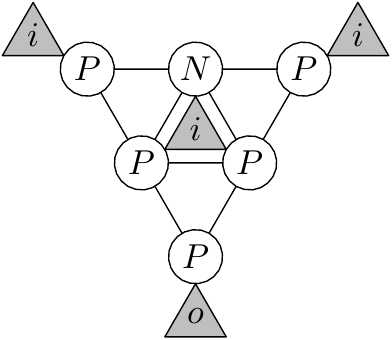}
\caption{The two dimensional diagram corresponding to $p =
x_1^3+3x_1x_2+x_2^3$ with sinks and sources marked respectively $i$ and $o$.}\label{fig:ss}
\end{center}
\end{figure}
See Figure~\ref{fig:ss} for a sample two dimensional diagram with sinks and
sources marked.  In the diagram sinks and sources are marked with a
triangle, with $i$ for sink and $o$ for source.  The vertices of the triangle
mark the points that affect the sink or source.  The top vertex marks the
$\alpha$ and the bottom two vertices of the triangle mark
$\operatorname{down}_1(\alpha)$ and
$\operatorname{down}_2(\alpha)$.

Note that the monomials
$x^{\alpha}$,
$x^{\operatorname{down}_1(\alpha)}$,
\ldots,
$x^{\operatorname{down}_n(\alpha)}$ in $q$ are those monomials that when
multiplied by $(s-1)$ are the ones that contribute to the monomial
$x^{\alpha}$ in $p-1 = q(s-1)$.  Therefore,
if $\alpha$ is a sink, then $x^{\alpha}$ in $p$ must have a positive coefficient.
Similarly, if $\alpha$ is a source then $x^{\alpha}$ in $p$ must have a
negative coefficient.
The following proposition is then immediate.

\begin{prop} \label{prop:ptodiag}
Let $p$ be a polynomial such that $p-1 = (s-1)q$.  Let $D$ be the
corresponding Newton diagram.  If $\alpha \in \Z^n$ is a sink, then
the monomial $x^\alpha$ has a nonzero positive coefficient in the polynomial $p-1$.
If $\alpha \in Z^n$ is a source then the monomial $x^\alpha$ has a nonzero negative
coefficient in $p-1$.  In particular,
\begin{equation}
\#(D) \leq N(p-1) \text{ which implies } \#(D)-1 \leq N(p) .
\end{equation}
Furthermore, if $p \in \sH$ (that is $p$ has all positive coefficients) and $p$ is nonconstant,
then $D$ has a unique source at $\alpha = (0,\ldots,0)$, and otherwise $D$ has only sinks.
\end{prop}

In \cite{LP} we have further symmetrized the problem to avoid the change in
signs for sinks and sources.  In most of the proofs in this paper the symmetric version
is not necessary and may in fact obscure the main argument.
Certain parts of the argument, however, will require the symmetric
version in two dimensions, and we therefore make the following definition.

\begin{defn}
Fix $d$.  Let
\begin{equation}
H_d = \{ (a,b,c) \in \Z^3 : a+b+c = d, \text{ and } a,b,c \geq 0 \} .
\end{equation}
A \emph{2-dimensional symmetric diagram} is a function
\begin{equation}
D \colon H_d \to \{ P, N, 0 \} .
\end{equation}

We call $(a,b,c) \in \Z^3$ a \emph{symmetric node} or just \emph{node}
if $D$ is nonzero on at least one
of $(a-1,b,c)$, $(a,b-1,c)$, and $(a,b,c-1)$, and $D$ does not take both
the values $P$ and $N$ on
$(a-1,b,c)$, $(a,b-1,c)$, and $(a,b,c-1)$.  In other words, $(a,b,c)$ is a
node if the direct image
\begin{equation}
D\bigl( \{ (a-1,b,c), (a,b-1,c), (a,b,c-1) \} \bigr)
= \{ 0, P \} \text{ or } \{ 0, N \} .
\end{equation}
As before let $\#(D)$ denote the number of nodes in $D$.
\end{defn}

\begin{prop}
Let $Q(x_1,x_2,x_3)$ be a homogeneous polynomial:
\begin{equation}
Q(x_1,x_2,x_3) =
\sum_{\abs{\alpha} = d} C_\alpha x^\alpha
\end{equation}
and define
\begin{equation}
D(\alpha) :=
\begin{cases}
P & \text{ if $C_{\alpha} > 0$, } \\
0 & \text{ if $C_{\alpha} = 0$, } \\
N & \text{ if $C_{\alpha} < 0$. }
\end{cases}
\end{equation}
Then $D$ is a 2-dimensional symmetric diagram and
\begin{equation}
P(x_1,x_2,x_3) = (x_1+x_2+x_3) Q(x_1,x_2,x_3)
\end{equation}
has at least $\#(D)$ distinct monomials.
\end{prop}

If we start with a $Q$ as above and look at
$q(x_1,x_2) = Q(x_1,x_2,-1)$, then $(x_1+x_2+x_3)$ becomes $(x_1+x_2-1)$ and we
are in the nonsymmetric setup.
We obtain a (nonsymmetric)
2-dimensional diagram by mapping $H_d$ into $\Z^2$ and flipping signs
for points corresponding to odd powers of $x_3$.  Sinks and sources
correspond to the symmetric nodes in the symmetric diagram.

The arguments in \cite{LP} relied heavily on a geometric interpretation of
the Newton diagram of $q$. Our terminology treats the Newton diagram as a
solid object formed by simplices at each of $\alpha \in \mathbb{Z}^n$ with
$D(\alpha) \neq 0$. The definitions of corners, edges and faces are made
with this picture in mind.  See Figure~\ref{fig:face2} for the geometric
intuition.

Our naming convention changes slightly from \cite{LP} so
that the naming is logical and consistent for diagrams of all
dimensions.

\begin{figure}[h!t]
\begin{center}
\includegraphics{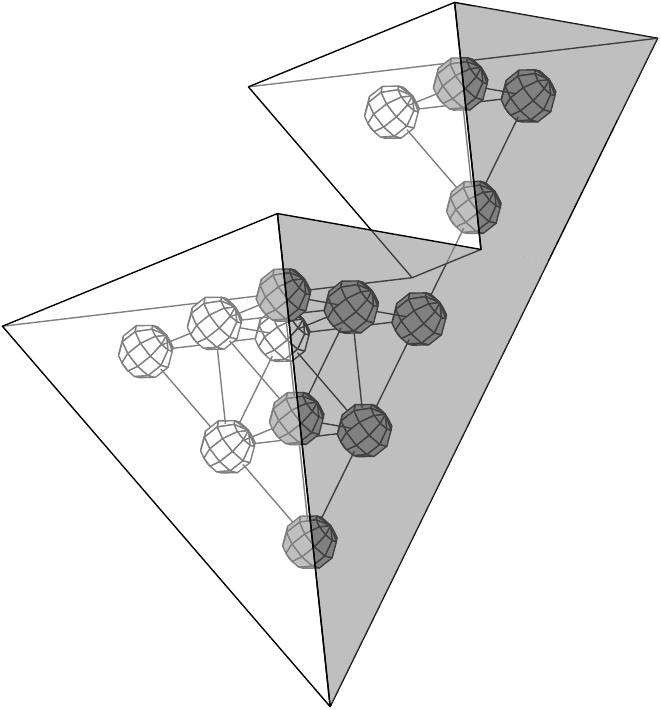}
\caption{Treating a diagram as a solid for $n=3$.  A vertical face and
its corresponding points are
shaded.}\label{fig:face2}
\end{center}
\end{figure}

\begin{defn}
An $\alpha \in \Z^n$ is called a \emph{top-corner} of a Newton
diagram $D$ if $D(\alpha) = 0$ and if there exists exactly one $k \in [1 ,
n]$ with $D\bigl(\operatorname{down}_k (\alpha)\bigr) \neq 0$. An $\alpha \in \Z^n$ is
called a \emph{bottom-corner} of a Newton diagram $D$ if $D(\alpha) \neq 0$
and if $D\bigl(\operatorname{down}_k (\alpha)\bigr) = 0$ for all $k \in [1 , n]$.

An $\alpha \in \Z^n$ is said to lie on a \emph{horizontal edge} of $D$ if
$D(\alpha) = 0$ and $D\bigl(\operatorname{down}_k (\alpha)\bigr) \neq 0$ for
exactly two distinct $k \in [1, n]$.

An $\alpha \in \Z^n$ is said to lie on a \emph{vertical edge} of $D$ if
$D(\alpha) \neq 0$ and $D\bigl(\operatorname{down}_k (\alpha) \bigr) \neq 0$
for exactly one $k \in [1, n]$.

Finally, an $\alpha \in \Z^n$ is said to lie in the \emph{interior of a
vertical face} of $D$ if $D(\alpha) \neq 0$ and
$D\bigl(\operatorname{down}_j (\alpha)\bigr) \neq 0$ for exactly two
distinct $k \in [1, n]$, and $\alpha$ lies in the interior of a
\emph{horizontal face} of $D$ if $D(\alpha) = 0$ and
$D\bigl(\operatorname{down}_k (\alpha)\bigr) \neq 0$ for exactly three
distinct $k \in [1, n]$.

If a node $\alpha \in \Z^n$ lies on a corner or edge or in the interior of a vertical or horizontal face then we call $\alpha$ respectively a corner-, edge- or facial-node. We often distinguish between top-corner nodes and bottom-corner nodes, and sometimes between vertical-edge nodes and horizontal-edge nodes.
\end{defn}

Two points $\alpha$ and $\beta$ are said to be \emph{adjacent} if there is
a $k$ such that $\alpha_j = \beta_j$ for all $j\not= k$ and
$\alpha_k-\beta_k = \pm 1$.
The subset $S \subset \Z^n$ is said to be \emph{connected} if for each $\alpha,
\beta \in S$ there is a path from $\alpha$ to $\beta$ along adjacent
elements of $S$. A vertical ($2$-dimensional) face of
$D$ is a maximal connected subset of the elements $\alpha \in \Z^n$ that are
either corners, lie on edges or lie in a vertical face, and that all lie in
the same $2$-dimensional plane in $\Z^n$ defined by $n-2$ equations of the
form $\alpha_j = c_j$. Similarly, a horizontal face of $D$ is a maximal
connected subset of the elements $\alpha \in \Z^n$ that are either corners,
lie on a horizontal edge or in a horizontal face, and that all lie in the
same $2$-dimensional plane in $\Z^n$ defined by $|\alpha| = d$ and $n-3$
equations of the form $\alpha_j = c_j$. For a vertical face there
are exactly two parameters $\alpha_k$ and $\alpha_m$ that can vary. We say
that the vertical face corresponds to the variables $x_k$ and $x_m$.

An important step in the proof of the $n=3$ case of Theorem
\ref{thm:monomialdegbnd}, proved in \cite{LP}, was to count the number of
nodes that lie in a $2$-dimensional face, and then add the numbers for all
the faces. A difficulty that arises when counting nodes in this manner is
that the edge-nodes and corner-nodes are counted multiple times, hence the
distinction between these nodes and facial nodes, which lie on only one
face and are therefore only counted once. A similar idea comes up in the
proof for $n \geq 4$, which we present in the next two sections.

%%%%%%%%%%%%%%%%%%%%%%%%%%%%%%%%%%%%%%%%%%%%%%%%%%%%%%%%%%%%%%%%%%%%%%%%%%%%%

\section{Views and Sides}

In a 2-dimensional diagram we say that $\alpha$ and $\beta$
are in the same \emph{row} if $\abs{\alpha} = \abs{\beta}$, that is
if they correspond to monomials of the same degree.

\begin{defn}
We define a \emph{simple diagram} to be a 2-dimensional diagram
with bottom nodes only in its lowest nonzero row.  The \emph{height} of a
simple diagram is the number of nonzero rows.
\end{defn}

In all arguments below
we simply ignore the bottom nodes of simple diagrams.

For each pair $(x_k, x_m)$, where $k \neq n$, in an $n$-dimensional diagram
$D$ of size $d$,
let us consider the corresponding vertical
faces and call this the \emph{side} of the diagram.  For each side
we can find a finite set of simple diagrams $\sF = (F_1,F_2,\ldots,F_s)$
such that the top and side edges of each $F_j$ correspond to top and
side edges of the diagram $D$.  Furthermore the heights of $F_i$
sum up to $d$.  Each edge node, top-corner node, and
facial node of $F_i$ corresponds to a node in $D$.  See Figure~\ref{fig:face} for
an example when $n=3$.  To find the $(F_1,\ldots,F_s)$, start at
the unique bottom node of $D$.  Then find the face $F_1$ containing
this bottom node
corresponding to the edge $(x_k,x_m)$.  Suppose that
$F_1$ has height $h_1$.  At the row directly above the top
row of $F_1$ (row $h_1+1$ in $D$),
find a nonzero point with the smallest degree in the
$x_1,\ldots,\hat{x}_k,\ldots,\hat{x}_m,\ldots,x_n$ variables
(where the $\hat{\cdot}$ means we are excluding that variable).  Find
the corresponding face and consider only that part of the face from row
$h_1+1$ upwards and mark the corresponding simple diagram $F_2$.
Continue until we reach the top row of $D$.  Depending on the points chosen
it is possible to get a different set of simple diagrams corresponding to
the edge $(x_k,x_m)$.  However, it suffices to fix one set of
simple diagrams corresponding to each edge.

\begin{figure}[h!t]
\begin{center}
\includegraphics{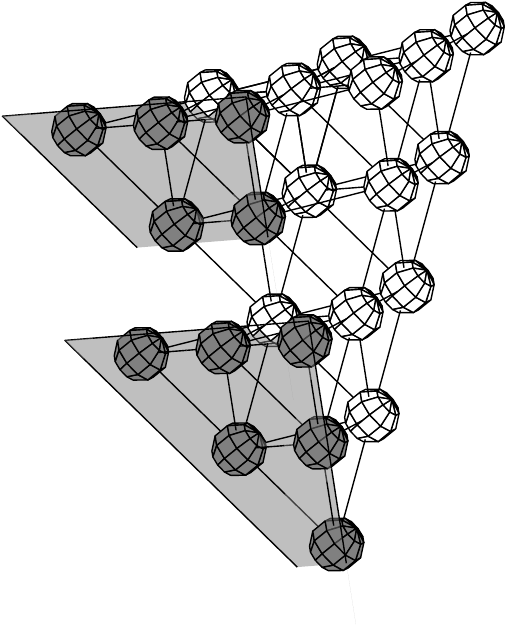}
\caption{Two simple diagrams $F_1$ and $F_2$ as faces corresponding to an edge
in a three dimensional diagram.}\label{fig:face}
\end{center}
\end{figure}

\begin{defn}
We call $\sF = (F_1,\ldots,F_s)$ a \emph{complete set of simple diagrams
corresponding to the edge $(x_k,x_m)$}.
\end{defn}

\begin{lemma} \label{dim2countlemma}
Let $D$ be a simple diagram of height $d$.  Let $f$ be the number of
facial nodes, $e$ the number of edge nodes, and $c$ the number of corner
nodes (excluding bottom nodes).  Then
\begin{equation}
2f + e + c \geq d+1 .
\end{equation}
\end{lemma}

\begin{proof}
We claim that we can ``fill'' the diagram without increasing the sum $2f + e
+ c$, starting with the zeroes on the lowest nonzero row, and working our
way up to the highest nonzero row of the diagram. By a filled diagram we
mean that the nonzero points in the lowest row are connected, and that
up to the highest row of the diagram, a nonzero point has only nonzero
points above it. Once the diagram is filled, the claim follows from the following argument, which was also used in the proof of Lemma 4.2 in \cite{LP}.

We refer to the horizontal rows of our filled simple diagram by rows $1$ through $d$, starting with the lowest row and ending with the highest. For each row we count the number of sign changes in the row as we move from one end to the other. We denote this number by $s_j$, the number of sign changes, where $j$ refers to the number of the row. For example, if the lowest row consists of only $P$-points then $s_1 = 0$.

Now consider rows $j$ and $j+1$, and note that if $s_j$ differs from $s_{j+1}$ then there must be at least one facial or edge node consisting of points in these two rows. In fact, note that $2$ times the number of facial nodes plus the number of edge nodes that occur here must be at least as large as $|s_{j+1} - s_j|$. Hence the total number of nodes (counting facial nodes for 2) in the diagram without counting the edge and corner nodes on the top row must be at least $|s_d - s_1|$. The number of edge or corner nodes on the top row equals exactly $l_d + 1 - s_d$, where $l_j$ denotes the length of the $j$-th row. Hence the total number of nodes is at least
\begin{equation}
l_d + 1 - s_d + |s_d - s_1| \ge l_d + 1 - s_d \ge l_d + 1 - (l_1 - 1) \ge d+1,
\end{equation}
which gives the desired inequality.

In order to prove our claim that we can always fill the diagram without increasing the number of nodes, we need to consider several cases. We start with the
lowest nonzero row. Suppose the $P$- and $N$-points on this row do not form
a connected subset of $\Z^n$. By the definition of a simple diagram there
are no higher bottom-nodes, which means that there must exist a sequence of
$0$-points on the lowest row that is enclosed on either side by nonzero points.
We fill
this sequence with alternating $P$'s and $N$'s. Of course such a filling
can be done
in two distinct ways, and it turns out that for at least one way the sum $2f
+ e + c$ does not increase. To see this fact, note that the only way the count
can be increased is when an edge-node above one of the most outer $0$-points
is changed into an facial node. However, if this can happen then we can
choose what sign we change the corresponding $0$-point to,
such that the edge-node
disappears. In this case the $(2f + e + c)$-count decreases by $1$ on
that outer $0$-point, and as it can increase by at most $1$ on the other side the
sum $2f + e + c$ does not increase. See Figure~\ref{fig:fill2d1}.

\begin{figure}[h!t]
\begin{center}
\includegraphics{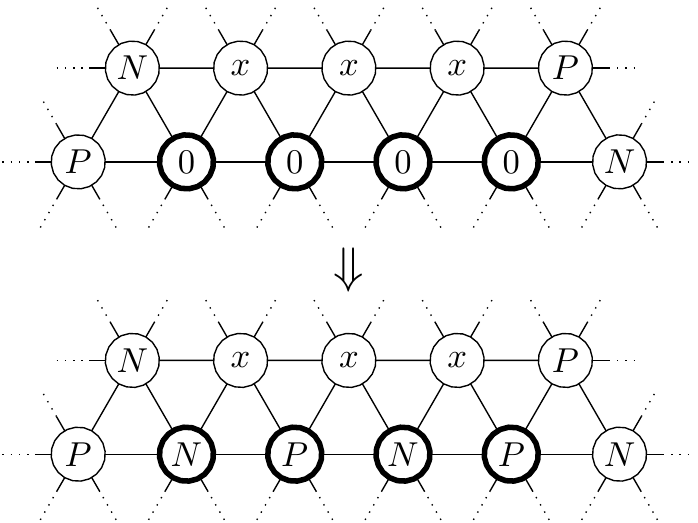}
\caption{Filling in the lowest row by a correctly chosen sequence of $P$'s
and $N$'s.  The exact values of the points marked with $x$'s is not relevant.}\label{fig:fill2d1}
\end{center}
\end{figure}

The argument for higher rows is similar but slightly different cases need to
be considered. We consider the lowest row for which there exist zeroes that
lie above nonzero points. If there are no nonzero points on this row then we
are done filling, so let us suppose that there exist nonzero points in this
row. Consider a $0$-point on
this row that lies above at least one nonzero point, and that is adjacent to a
nonzero point. We first consider the case when there are nonzero points on
both sides of this $0$-point. In this case we again have that by assigning
an $N$ or $P$ to this $0$-point, we can only increase the count if an
edge-node turns into an facial node. However, each time this can happen, the
edge node disappears if we assign the other sign to the $0$-point. By
choosing the sign for which the number of edge-nodes that changes into a
facial node is minimal, the $(2f + e + c)$-count does not increase.

The last case we need to consider is when the $0$-point lies next to exactly
one nonzero point. By changing this $0$-point into an $N$- or $P$-point we
may not only increase the $2f + e + c$ count by changing an edge-node into a
facial node, but we also create a new corner-node which increases the
count by $1$. If the number of edge-nodes that can change to a facial node
is odd (in which case it necessarily is $1$), then we can always choose the
sign so that more edge nodes disappear than there are changed into facial
nodes and we are done. So let us assume that the number of edge nodes that
can change into facial nodes is even, either $0$ or $2$. In this case we
fill a horizontal list of $0$-nodes, starting with the given $0$-node and
ending with the nearest $0$-node that is either adjacent to a nonzero node,
or that is part of a node itself. We fill this list with alternating $P$- and
$N$-points, see Figure \ref{fig:fill2d2} for an example. Note that in this example we cannot fill fewer $0$-points without increasing the number of nodes.

\begin{figure}[h!t]
\begin{center}
\includegraphics{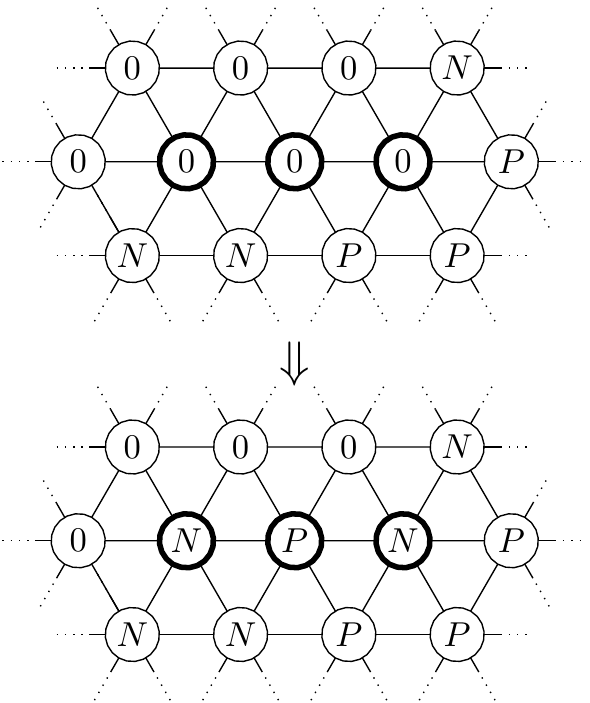}
\caption{Filling in the rightmost three $0$-points by alternating $P$- and $N$-points.}\label{fig:fill2d2}
\end{center}
\end{figure}

When the nearest $0$-point is adjacent to a nonzero
node we no longer create a new corner-node and we again choose the sign so
that at least as many edge-nodes are removed as there are edge-nodes changed
into facial nodes and we are done. In the case when the nearest $0$-point is
not adjacent to a nonzero point but is a node itself, then by filling that
point we can change the $(2f + e + c)$-count by exactly $1$. By our
assumption the number of nodes that is changed for the $0$-point that we
started with was even, so the total is odd. Therefore we can always choose
the sign of the alternating list such that more edge- and corner-nodes
disappear than edge-nodes are changed into facial nodes. The gain obtained
from the new corner-node can therefore be countered by choosing the proper
sign. This completes the argument.
\end{proof}

\begin{defn}
Let $D$ be an $n$-dimensional diagram.
For each $\alpha \in \Z^{n-1}$ define
\begin{equation}
\gamma_{k,m}(\alpha,a,b) := (\alpha_1,\ldots,\alpha_{k-1},a,\alpha_{k+1},\ldots,\alpha_{m-1},b,\alpha_{m},
\alpha_{m+1},\ldots,\alpha_{n-1}) \in \Z^n ,
\end{equation}
where $1 \leq k \leq n-1$, $1 \leq m \leq n$, and $k \not= m$.  That is, we replace the
$k$th element with $a$ and we prepend $b$ before the $m$th element.  If
$m=n$ then we append $b$ onto the end.  E.g.\ $\gamma_{2,3}\bigl( (0,2), a,
b \bigr) = (0,a,b)$.

Let
\begin{equation}
\gamma_{k,m}(\alpha)
:=
\gamma_{k,m}(\alpha,a,b) ,
\end{equation}
where $a+b = \alpha_k$, and $a$ is the smallest integer such that
$D\bigl(\gamma_{k,m}(\alpha,a,b) \bigr)$ is nonzero.

Define $D' = V(D,k,m)$ as an
$(n-1)$-dimensional diagram by setting
\begin{equation}
D'(\alpha) :=
D\bigl(\gamma_{k,m}(\alpha) \bigr).
\end{equation}

We say that a node $\alpha \in \Z^{n-1}$ of $D'$ corresponds to a node
$\beta$ of $D$ if $\beta = \gamma_{k,m}(\alpha)$.
\end{defn}

Intuitively, $V(D,k,m)$ is the ``view'' along the edge looking from
$x_k$ towards $x_m$.  See Figure~\ref{fig:view} for an example in $n=3$.

\begin{figure}[h!t]
\begin{center}
\includegraphics{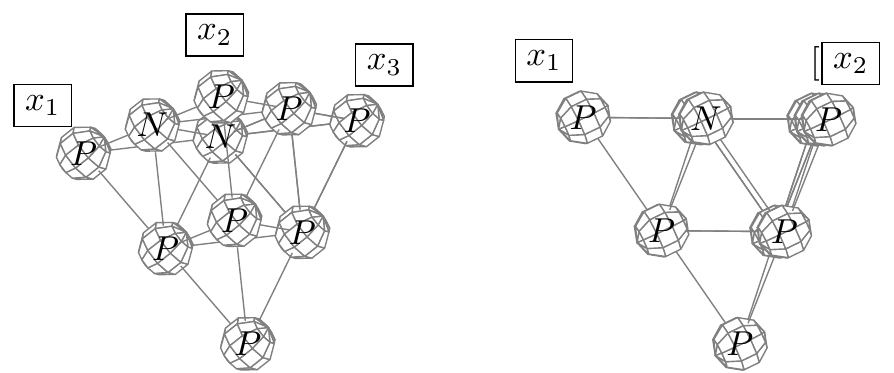}
\caption{Example ``view'' of a diagram.  On the left is a 3-dimensional
diagram $D$ with directions labeled by the corresponding variables.
On the right we have the 2-dimensional $V(D,2,3)$.}\label{fig:view}
\end{center}
\end{figure}

It is not hard to see that if $\alpha \in \Z^{n-1}$ is a node of $V(D,k,m)$
then it always corresponds to a node $\gamma_{k,m}(\alpha)$ of $D$.
Thus we have the rather simple estimate
\begin{equation}
\#(V(D,k,m)) \leq \#(D) .
\end{equation}
In fact, if $\alpha$ is a sink (resp.\ source), then
$\gamma_{k,m}(\alpha)$ is a sink (resp.\ source).  Hence we also
have the following proposition.

\begin{prop}
If $D$ has a unique source at the origin $\alpha=(0,\ldots,0)$ and
all other nodes are sinks, then $V(D,k,m)$ has a unique source at
the origin and all other nodes are sinks.
\end{prop}

When diagrams have only a single source at the origin, they have a special
form.  We define two conditions that we need in the proof.

\begin{defn}
A set $K \subset \Z^2$ has \emph{left overhang} if there exists
a point $(a,b) \in K$, $(a,b) \not= (0,0)$ such that
$(a,b-1) \notin K$ and $(a-1,y) \notin K$ for all $y \geq b$.

A point $(a,b) \in K$ is \emph{right overhang} if it is a left overhang
after swapping variables.  We say simply that $(a,b)$ is an \emph{overhang}
if it is a left or a right overhang.

Let $K \subset \Z^n$ be the support of a diagram $D$, that is
let $K = D^{-1}(\{P,N\})$.  Define a 2-dimensional projection
\begin{equation}
\pi(K,k,m)
: =
\{ (a,b) \in \Z^2 : \text{ there exists } \alpha \in K \text{ with }
\alpha_k= a \text{ and } \alpha_m = b \} .
\end{equation}

We say that $D$ has \emph{no overhang} if $\pi(K,k,m)$ has no overhang for every $k,m = 1,\ldots,n$,
$k\not=m$, where $K$ is the support of $D$.
\end{defn}

The above definition of no overhang agrees precisely with the definition
from \cite{LP}.  See Figure~\ref{fig:overhang} for an example.

\begin{figure}[h!t]
\begin{center}
\includegraphics{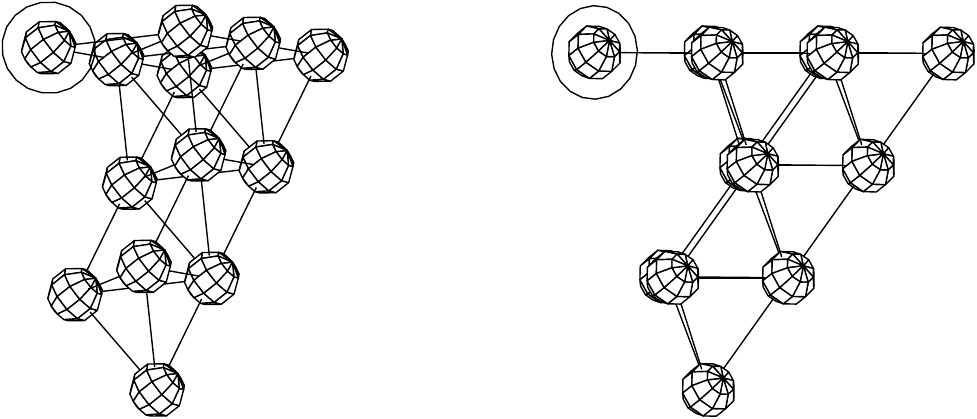}
\caption{Diagram showing an example of overhang in
$n=3$.}\label{fig:overhang}
\end{center}
\end{figure}

\begin{defn}
Let $D$ be an $n$-dimensional diagram ($n \geq 2$).
A set $E \subset D^{-1}(\{P,N\}) \subset \Z^n$ (subset of the support
of $D$) is an \emph{outside vertical edge} if
there exists $\alpha \in \Z^n$ and an integer $k$ such that
\begin{equation}
E = \{
\operatorname{down}_k(\alpha),
\operatorname{down}_k\bigl(
\operatorname{down}_k(\alpha)\bigr),
\ldots,
\operatorname{down}_k\bigl(
\operatorname{down}_k\bigl( \cdots \operatorname{down}_k(\alpha)
\cdots \bigr)\bigr) \} ,
\end{equation}
and such that if $D(\alpha) = 0$,
$D\bigl(\operatorname{down}_j(\alpha)\bigr) = 0$ for all $j \not= k$,
and
$D\bigl(\operatorname{down}_j(\beta)\bigr) = 0$ for all $j \not= k$
and all $\beta \in E$.
See Figure~\ref{fig:outsideedge}.

If $\beta \in E$ and $D\bigl(\operatorname{down}_k(\beta)\bigr) = 0$,
then we say $\beta$ is a \emph{bottom node} (it is in fact a node).

We say that a diagram $D$ has \emph{no outside vertical edge nodes}
if for every outside vertical edge $E$, no element $\beta \in E$
is a node, except possibly the bottom node if $E$ contains one.
\end{defn}

\begin{figure}[h!t]
\begin{center}
\includegraphics{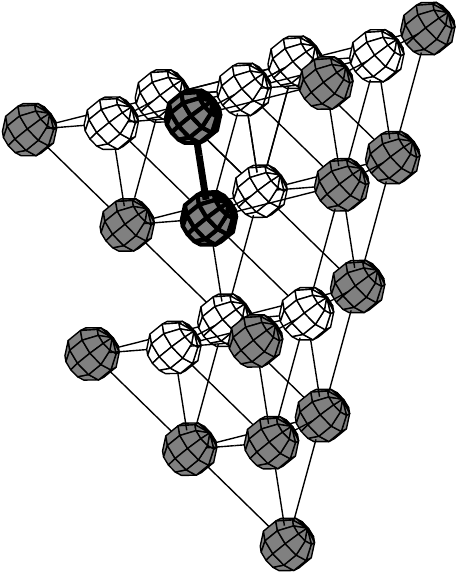}
\caption{Diagram with outside edges marked by dark balls.  One specific
outside edge is marked with thick lines.}\label{fig:outsideedge}
\end{center}
\end{figure}

\begin{prop} \label{overandvertprop}
If $D$ is an $n$-dimensional ($n \geq 2$)
diagram with a unique source at the origin (all other nodes
are sinks) then $D$ has no overhang and no outside vertical edge nodes.
\end{prop}

\begin{proof}
First, suppose that $D$ has an outside vertical edge node, then it has to be
a sink at $\beta \in E$ for some outside vertical edge $E$.
Then $D(\beta) = N$ as $D(\beta)$ cannot be 0.  If $\alpha$ and $k$ are
as in the definition of the edge $E$, then $\alpha$ must be a node, and by our hypothesis on $D$ it has to
be a sink.  As $D(\alpha) = 0$, then
$D\bigl(\operatorname{down}_k(\alpha)\bigr) = P$, but then
there must be some source on the edge between $\alpha$ and $\beta$
contradicting the hypothesis on $D$.

Suppose that there is an overhang.  We note that the 2-dimensional
projection of the support in the definition of no overhang is the support
of some diagram
\begin{equation}
D' =
V\bigl(\cdots V(V(D,k_1,k_2),k_3,k_4),\ldots,k_{m-1},k_m\bigr),
\end{equation}
for some sequence of integers $k_1,\ldots,k_m$ taking views along distinct
edges until we arrive at a 2-dimensional diagram $D'$ that has an
overhang at some $(a,b) \in \Z^2$.  By definition of overhang,
$(a,b)$ is the bottom node of an outside vertical edge while not being
the origin.  The outside vertical edge must consist of $P$'s only
otherwise we would obtain a source in the diagram, but if
$D\bigl((a,b)\bigr) = P$, then $(a,b)$ is a source.  We have a contradiction
as $(a,b)$ was not the origin.
\end{proof}

The definition of $V(D,k,m)$ can be used even for a symmetric 2-dimensional
diagram.  The result is a 1-dimensional symmetric diagram, that is, simply a
single row of $P$'s, $N$'s, and $0$'s.  A node on such a 1-dimensional
diagram is then simply two points next to each other that are of the
configuration
$(0,N)$, $(N,0)$, $(N,N)$,
$(0,P)$, $(P,0)$, $(P,P)$.  In other words, a node is two points next to
each other not in the configuration $(0,0), (P,N)$ or $(N,P)$.

\begin{prop} \label{prop:horizhid}
Let $D$ be a 2-dimensional symmetric diagram.
Either
\begin{enumerate}[(i)]
\item $\#(D) = 3$, in which case $D$ has one nonzero point, or
\item $\#(D) > 3$, in which case size of $D$ is greater than one, and
there exists a view $V(D,k,m)$ (where $k,m = 1,2,3$), such that
\begin{equation}
\#\bigl(V(D,k,m)\bigr) \leq \#(D) - 2 .
\end{equation}
\end{enumerate}
\end{prop}

\begin{proof}
In \cite{LP} the authors have shown that if $\#(D) = 3$, then $D$ is of size
one, that is $D$ is a single point.

In case $\#(D) > 3$, first suppose that the support of
every $V(D,k,m)$ is connected, that is, there is exactly one node of the
form $(0,*)$ and exactly one node of the form $(*,0)$, where $*$ is $P$ or
$N$.  If every $V(D,k,m)$ has only those two nodes, then
$\#\bigl(V(D,k,m)\bigr) \leq \#(D) - 2$ for every $V(D,k,m)$.
Thus suppose that
$V(D,k,m)$ is such that there is a node $\alpha$ in $D$ that corresponds
to a node of the form $(P,P)$ or a node $(N,N)$
in $V(D,k,m)$.  Then $\alpha$ does not correspond to any node in
$V(D,m,k)$, see Figure~\ref{fig:2d1dview}.  It is easy to see that at least one node of the form
$(0,*)$ or $(*,0)$ gets hidden by
$V(D,m,k)$.  Thus
$\#\bigl(V(D,m,k)\bigr) \leq \#(D) - 2$ for every $V(D,m,k)$.

\begin{figure}[h!t]
\begin{center}
\includegraphics{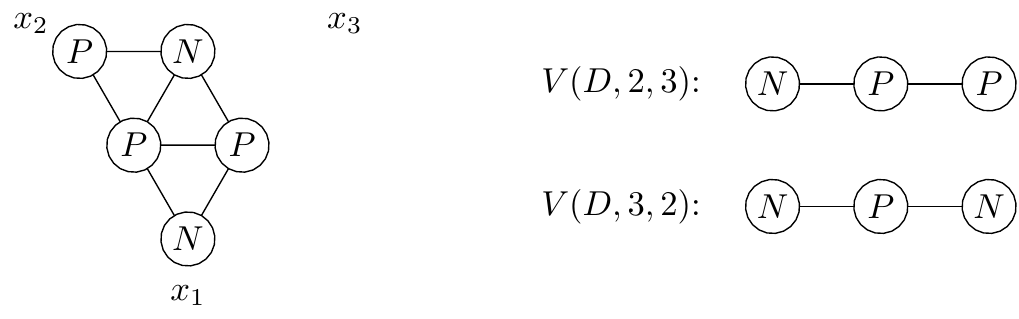}
\caption{A 2-dimensional symmetric diagram with directions labeled by the
corresponding variables.  Note that $V(D,2,3)$ contains a node of the form
$(P,P)$ that does not correspond to any node (gets hidden) in $V(D,3,2)$.}\label{fig:2d1dview}
\end{center}
\end{figure}

If the support of some $V(D,k,m)$ is not connected, the proposition follows
easily by noting that there are two nodes in $D$
corresponding to nodes of the form $(*,0)$ or $(0,*)$ in $V(D,k,m)$ that
get hidden.
\end{proof}

%%%%%%%%%%%%%%%%%%%%%%%%%%%%%%%%%%%%%%%%%%%%%%%%%%%%%%%%%%%%%%%%%%%%%%%%%%%%%

\section{Proof of the main result}

The main step towards the proof of Theorem \ref{thm:monomialdegbnd} is taken
in the following lemma.  By a \emph{top horizontal face} we mean a
horizontal face of the diagram corresponding to the highest degree.

\begin{lemma} \label{indsteplemma}
Let $D$ be an $n$-dimensional, $n \geq 4$, diagram of size $d$
with a unique source at the origin.
Then there exist
integers $k, m$ ($k \not= m$, and
$1 \leq k,m \leq n$) such that
\begin{equation} \label{eq:stepbound}
\#\bigl(V(D,k,m)\bigr) \leq \#(D) - d .
\end{equation}
Furthermore, if any complete set of simple diagrams corresponding to
any edge
contains a facial node,
or if any top horizontal face has more than three nodes, then
we can choose $k$ and $m$ such that
\begin{equation} \label{eq:stepbound2}
\#\bigl(V(D,k,m)\bigr) < \#(D) - d .
\end{equation}
\end{lemma}

\begin{proof}
For each edge we look at all possible corresponding complete sets of simple diagrams, and count the number of facial nodes in these complete sets. We select an edge $(x_k, x_m)$ that has a corresponding complete set of simple diagrams $\sF$ with the minimal number of facial nodes.

In the view $V(D,k,m)$, we note that no facial node on $\sF$
corresponds to a node in $V(D,k,m)$.  Similarly, no edge node corresponds
to a node in $V(D,k,m)$ since only outside vertical edge nodes could possibly be
seen in $V(D,k,m)$ and no such nodes exist by
Proposition~\ref{overandvertprop}.
Let $F_1,\ldots,F_s$ be the simple diagrams contained in $\sF$ as described
above.  If $d_j$ is the height of $F_j$ then
$\sum d_j = d$.  A corner node in $F_j$ that is not a top-corner of maximal
degree (maximal degree in $F_j$) is not visible in $V(D,k,m)$.
Were such a node visible in $V(D,k,m)$, it would be an outside
vertical edge node and $V(D,k,m)$ has no such nodes by
Proposition~\ref{overandvertprop}.  Of the top-corner
nodes of maximal degree, we can ``see'' exactly one of them (the one
corresponding to the $x_k$ corner), hence one top-corner node is visible.

Let $e_j$, $c_j$, and $f_j$ denote the edge, corner and facial nodes in $F_j$
(not counting any bottom corner nodes).
First suppose that there are no facial nodes in any
complete set of simple diagrams corresponding to any edge.
Then by Lemma~\ref{dim2countlemma}, we see that there must be
at least $e_j + c_j \geq d_j+1$ nodes in $F_j$ and at most one of them is
visible.  Thus
\begin{equation}
\sum_{j=1}^s (e_j + c_j)
\geq
\sum_{j=1}^s (d_j + 1)
=
d+s .
\end{equation}
Hence there are at least $d+s$ nodes in $\sF$ and of those only $s$
are visible in $V(D,k,m)$.  Thus $d$ nodes in $D$ do not correspond
to any node in $V(D,k,m)$ and \eqref{eq:stepbound} holds.

Still assume that there are no facial nodes in any
complete set of simple diagrams corresponding to any edge.
Suppose that there is some top horizontal face with more than
3 nodes.  Without loss of generality let this be the horizontal face (or
faces)  corresponding to the vertices $\{ x_1 , x_2, x_3 \}$.
Then there is at least one view
$V(D,1,2)$, $V(D,1,3)$ or $V(D,2,3)$ that must hide at least two nodes by
Proposition~\ref{prop:horizhid}.
But in the above estimate we only
counted one hidden node in the top degree.  Hence for that view, at least
$d+1$ nodes are hidden and so \eqref{eq:stepbound2} holds.

Now suppose that there are facial nodes in $\sF$. There are at least
two other edges with one vertex being $x_m$.  Let $\sF'$ and $\sF''$
be the corresponding complete sets of simple diagrams.
As $\sF$ had the least number of facial nodes,
then $\sF'$ and $\sF''$ have at least as many facial nodes as $\sF$.

We observe that
as the view goes in the direction from $x_k$ to $x_m$, any facial nodes in $\sF'$
and $\sF''$ do not correspond to nodes in $V(D,k,m)$.  Hence
for each facial node in $\sF$, two other facial nodes are hidden, one in $\sF'$
and one in $\sF''$.  Therefore, by the above calculation, for each $F_j$
there are at least
\begin{equation}
3 f_j+ e_j + c_j - 1
\end{equation}
nodes hidden in $V(D,k,m)$.  As $2 f_j+ e_j + c_j \geq d+1$, then
if $f_j > 0$ for any $j$, more than $d$ nodes are hidden in $V(D,k,m)$
and \eqref{eq:stepbound2} holds.
\end{proof}

Lemma \ref{indsteplemma} is the induction step in the proof of Theorem
\ref{thm:monomialdegbnd}. The basis for this induction is the following
result, which was proved by the authors in~\cite{LP}.

\begin{thm}[see \cite{LP}] \label{thm:diagbound3}
Let $D$ be a $3$-dimensional diagram of size $d$ with a unique source at
the origin.  Then
\begin{equation}
2d+2 \leq \#(D).
\end{equation}
\end{thm}

\begin{thm} \label{thm:diagbound}
Let $D$ be an $n$-dimensional, $n \geq 4$, diagram of size $d$
with a unique source at the origin.  Then
\begin{equation} \label{eq:diagbound}
(n-1)d+2 \leq \#(D) .
\end{equation}
\end{thm}

\begin{proof}
The proof follows by induction on $n$. We know the result holds for $n = 3$
by Theorem \ref{thm:diagbound3}.  Let $n \geq 4$ and suppose that the
result holds for dimension $n-1$.  By Lemma~\ref{indsteplemma} we can find a view $V$ with
$\#(V) \leq \#(D) - d$. As $V$ is $(n-1)$-dimensional we have that
\begin{equation}
(n-2)d + 2 \leq \#(V)
\end{equation}
and inequality \eqref{eq:diagbound} follows.
\end{proof}

\begin{thm} \label{thm:diagwhitney}
Let $D$ be an $n$-dimensional, $n \geq 4$, diagram of size $d$
with a unique source at the origin.  Suppose that $D$
minimizes
$\#(D)$ for the given degree $d$, that is
$(n-1)d+2 = \#(D)$.  Then $D$ contains precisely one point in each
degree.
\end{thm}

It is not hard to see that if a diagram $D$ contains precisely one point in each
degree such that for each degree precisely one of the possible nodes is
cancelled by the point in the higher degree (a diagram corresponding
to a sharp generalized Whitney polynomial), then we see that
$\#(D) = (n-1)d+2$.

\begin{proof}
If there were more than $3$ nodes on any top horizontal face,
then we could apply
Lemma~\ref{indsteplemma} to find a view $V$ with
$\#(V) < (n-1)d+2 - d = (n-2)d+2$, which would contradict the
bound of Theorems
\ref{thm:diagbound3} (if $n = 4$) and \ref{thm:diagbound} (if $n > 4$).

Thus all top horizontal faces have 3 nodes.  By the 2-dimensional bounds,
this means that the top horizontal faces must all consist of a single point,
therefore the top degree part of $D$ consists of isolated
points.  By ``isolated'' we mean that the corresponding horizontal faces (each
consisting of a single point) are disjoint.
It is not hard to see that removing any of
those points would get a diagram $D'$ with $n-1$ fewer nodes (only one
node can be created and $n$ nodes would be removed).  The diagram $D'$
would still have a unique source at the origin.  As $D$
minimizes $\#(D)$ for the given degree, then the degree of $D'$ must be one
less.  That is, there could have been at most one point in the top degree part
of $D$.  The proof follows by induction on the degree.
\end{proof}

As we have mentioned above, Proposition~\ref{prop:ptodiag}, a term in
$p \in \sH$ corresponds to a sink in the corresponding diagram, and in fact
$\#(D)-1 \leq N(p)$.  Theorem~\ref{thm:diagbound} thus proves the
main estimate in Theorem~\ref{thm:monomialdegbnd}.
Theorem~\ref{thm:diagwhitney} implies that
a $p \in \sH$ that satisfies equality in the bound, must be such
that $q = \frac{p-1}{s-1}$ has exactly one nonzero term of each degree.  It
follows immediately that such a $p$ must be a sharp generalized Whitney polynomial,
proving Theorem~\ref{thm:equality}.

\begin{remark}
We conjecture that the same degree bound holds for all $n$
without requiring that $p$
has positive coefficients, under the weaker
indecomposability condition discussed in \cite{LP}.  This conjecture was proved
for $n=2$ in \cite{LP}.  In the above the key assumptions
that make the proof work are no overhang, which was required in \cite{LP}
to prove the bound for $n=3$, and in the present proof we required no
outside vertical edge nodes for $n \geq 4$.  Both assumptions are
automatically satisfied for a diagram with a unique source at the origin.
We note that apart from these assumptions we did not differentiate between sinks and sources.
\end{remark}

%%%%%%%%%%%%%%%%%%%%%%%%%%%%%%%%%%%%%%%%%%%%%%%%%%%%%%%%%%%%%%%%%%%%%%%%%%%%%

%FIXME: else I don't get links, weird
%\renewcommand\MR[1]{\relax\ifhmode\unskip\spacefactor3000 \space\fi
  %\def\@tempa##1:##2:##3\@nil{%
    %\ifx @##2\@empty##1\else\textbf{##1:}##2\fi}%
  %\href{http://www.ams.org/mathscinet-getitem?mr=#1}{MR \@tempa#1:@:\@nil}}
\def\MR#1{\relax\ifhmode\unskip\spacefactor3000 \space\fi%
  \href{http://www.ams.org/mathscinet-getitem?mr=#1}{MR#1}}

\end{document}